\documentclass[a4paper, 12pt]{amsart}

\usepackage{amssymb,amsmath}
\usepackage[dvips]{graphics}
\usepackage[dvips]{color}
\usepackage{graphicx}
\usepackage[english]{babel}
\usepackage[latin1]{inputenc}
\usepackage{comment}
\usepackage{todonotes}

\newtheorem{theorem}{Theorem}[section]
\newtheorem{lemma}[theorem]{Lemma}

\newtheorem{corollary}[theorem]{Corollary}

\theoremstyle{definition}

\newtheorem{example}[theorem]{Example}
\newtheorem{remark}[theorem]{Remark}

\numberwithin{equation}{section}

\title{Regularity of the local fractional maximal function}

\thanks{This work is supported by the Academy of Finland.}

\author[Heikkinen, Kinnunen, Korvenp\"a\"a and Tuominen]
{Toni Heikkinen, Juha Kinnunen, Janne Korvenp\"a\"a and Heli Tuominen}

\newcommand\rn{\mathbb R^n}

\newcommand\ph{\varphi}
\newcommand\eps{\varepsilon}
\newcommand\M{\operatorname{\mathcal M}}

\newcommand\dist{\operatorname{dist}}

\providecommand{\ch}[1]{\text{\raise 2pt \hbox{$\chi$}\kern-0.2pt}_{#1}}

\providecommand{\vint}[1]{\mathchoice
          {\mathop{\vrule width 5pt height 3 pt depth -2.5pt
                  \kern -9pt \kern 1pt\intop}\nolimits_{\kern -5pt{#1}}}%
          {\mathop{\vrule width 5pt height 3 pt depth -2.6pt
                  \kern -6pt \intop}\nolimits_{\kern -3pt{#1}}}%
          {\mathop{\vrule width 5pt height 3 pt depth -2.6pt
                  \kern -6pt \intop}\nolimits_{\kern -3pt{#1}}}%
          {\mathop{\vrule width 5pt height 3 pt depth -2.6pt
                  \kern -6pt \intop}\nolimits_{\kern -3pt{#1}}}}
                  
\begin{document}


\begin{abstract}
This paper studies smoothing properties of the local fractional maximal operator, which is defined in a proper subdomain of the Euclidean space.
We prove new pointwise estimates for the weak gradient of the maximal function, which imply norm estimates in Sobolev spaces.
An unexpected feature is that these estimates contain extra terms involving spherical and fractional maximal functions.
Moreover, we construct several explicit examples which show that our results are essentially optimal. 
Extensions to metric measure spaces are also discussed.
\end{abstract}

\subjclass[2010]{42B25, 46E35}

\maketitle

\section{Introduction}
Fractional maximal operators are standard tools in partial differential equations, potential theory and harmonic analysis. In the Euclidean setting, they have been studied in \cite{A1}, \cite{A2}, \cite{AH}, \cite{KS}, \cite{KrKu}, \cite{LMPT} and \cite{MW}. It has been observed in \cite{KS} that
the global fractional maximal operator $\M_\alpha$, defined by
\begin{equation}\label{M}
 \M_\alpha u(x)=\sup_{r>0}\,r^{\alpha}\vint{B(x,r)}|u(y)|\,dy,
\end{equation}
has similar smoothing properties as the Riesz potential. More precisely, 
there is a constant $C$, depending only on $n$ and $\alpha$, such that 
\begin{equation}\label{global pointwise estimate for DM_alpha}
|D\M_{\alpha}u(x)|\le C\M_{\alpha-1}u(x)
\end{equation}
for almost every $x\in\rn$.
This implies that the fractional maximal operator maps $L^{p}(\rn)$  to a certain Sobolev space.
If the function itself is a Sobolev function, then the fractional maximal function belongs to a Sobolev space with a higher exponent. This follows quite easily from the Sobolev theorem using the facts that $\M_{\alpha}$ is sublinear and commutes with translations, see \cite[Theorem 2.1]{KS}. 
The regularity properties of the Hardy-Littlewood maximal function, that is \eqref{M} with $\alpha=0$, have been studied in \cite{AP}, \cite{B}, \cite{HM}, \cite{HO}, \cite{K}, \cite{Ko3}, \cite{Ku}, \cite{L}, \cite{L3} and \cite{T}.

This paper studies smoothness of the local fractional maximal function 
\[
\M_{\alpha,\Omega} u(x)=\sup\,r^{\alpha}\vint{B(x,r)}|u(y)|\,dy,
\]
where the supremum is taken over all radii $r$ satisfying $0<r<\dist(x,\rn\setminus\Omega)$. 
In this case, the family of balls in the definition of the maximal function depends on the point $x\in\Omega$ and the same arguments as in the global case do not apply. For the Hardy-Littlewood maximal function, 
the question has been studied in \cite{KL} and  \cite{HO}, see also \cite{L2}. 
For the local Hardy-Littlewood maximal operator $\M_{\Omega}$ with $\alpha=0$ we have
\begin{equation}\label{M1}
|D \M_{\Omega}u(x)|\le 2\M_{\Omega}|Du|(x)
\end{equation}
for almost every $x\in\Omega$.
In particular, this implies that the maximal function is bounded in Sobolev space $W^{1,p}(\Omega)$ when $1<p\le\infty$.

The situation is more delicate for the local fractional  maximal operator $\M_{\alpha,\Omega}$ with $\alpha>0$. 
One might expect that a pointwise estimates \eqref{global pointwise estimate for DM_alpha} and \eqref{M1} would also hold in that case.
However,  this is not true as such. 
Instead of \eqref{global pointwise estimate for DM_alpha}, we have
\[
|D \M_{\alpha,\Omega}u(x)|
\le C\big(\M_{\alpha-1,\Omega}u(x)+ \mathcal S _{\alpha-1,\Omega}u(x)\big)
\] 
for almost every $x\in\Omega$, where $C$ depends only on $n$. 
The local spherical fractional maximal function is defined as
\[
\mathcal S_{\alpha-1,\Omega}u(x) 
= \sup r^{\alpha-1} \vint{\partial B(x,r)}|u(y)|\,d\mathcal H^{n-1}(y),
\]
where the supremum is taken over all radii $r$ for which $0<r<\operatorname{dist}(x,\rn\setminus\Omega)$.
Norm estimates for the spherical maximal operator are much more delicate than the corresponding estimates for the standard maximal operator, but they can be obtained along the lines of \cite{S1} and \cite{SS}. 
These estimates are of independent interest and they are discussed in Section 2.
Consequently, the local fractional maximal function belongs locally to a certain Sobolev space. 

We also show that 
\[
|D\M_{\alpha,\Omega}u(x)|\le 2 \M_{\alpha,\Omega}|Du|(x) + \alpha \M_{\alpha-1,\Omega} u(x)
\]
for almost every $x\in\Omega$. This is an extension of \eqref{M1}, but again there is and extra term on the right hand side. Because of this the local fractional maximal function of a Sobolev function is not necessarily smoother than the fractional maximal function of an arbitrary function in $L^{p}(\Omega)$. This is in a strict contrast with the smoothing properties in the global case discussed in \cite{KS}.
Moreover, we show that $\M_{\alpha,\Omega}u$ has zero boundary values in the Sobolev sense and hence it can be potentially used as a test function in the theory of partial differential equations.
In Section \ref{sec: examples}, we construct several explicit examples, which complement our study and show that our results are essentially optimal. Another delicate feature is that the local fractional maximal operator over cubes
has worse smoothing properties than $\M_{\alpha,\Omega}$ defined over balls.

In the last section, we extend the regularity results of the local fractional maximal operator in metric measure spaces. As in the non-fractional case \cite{AK}, we use a discrete version of the maximal operator, because the standard maximal operators do not have the required regularity properties without any additional assumptions on the metric and measure. In the metric setting, fractional maximal operators have been studied for example in \cite{GGKK}, \cite{Go}, \cite{Go2}, \cite{HKNT}, \cite{HT}, \cite{PW}, \cite{SWZ} and \cite{W}.

\section{Notation and preliminaries}\label{sec: preliminaries}

Throughout the paper, the characteristic function of a set $E$ is denoted by $\ch{E}$.  
In general, $C$ is a positive constant whose value is not necessarily the same at each occurrence.  

Let $\Omega\subset\rn$ be an open set such that $\rn\setminus\Omega\ne\emptyset$ and let $\alpha\ge 0$. 
The local fractional maximal function of a locally integrable function $u$ is 
\[
\M_{\alpha,\Omega} u(x)=\sup\,r^{\alpha}\vint{B(x,r)}|u(y)|\,dy,
\]
where the supremum is taken over all radii $r$ satisfying $0<r<\dist(x,\rn\setminus\Omega)$.
Here 
\[
 \vint{B}u(y)\,dy=\frac1{|B|}\int_Bu(y)\,dy
\]
denotes the integral average of $u$ over $B$. If $\alpha=0$, we have the local Hardy-Littlewood maximal function 
\[
\M_{\Omega} u(x)=\sup\vint{B(x,r)}|u(y)|\,dy. 
\]
When $\Omega=\rn$, the supremum is taken over all $r>0$ and we obtain the fractional maximal function 
$\M_{\alpha}u$ and the Hardy-Littlewood maximal function $\M u$. A Sobolev type theorem for the fractional maximal operator follows easily from the Hardy-Littlewood maximal function theorem.

\begin{theorem}\label{fracM bounded rn}
Let $p>1$ and $0<\alpha<n/p$. There is a constant $C>0$, independent of $u$,
such that
 \[
 \|\M_{\alpha}u\|_{L^{p^*}(\rn)}
 \le C\|u\|_{L^p(\rn)},
 \]
for every $u\in L^{p}(\rn)$ with $p^*=np/(n-\alpha p)$. 
\end{theorem}
Now the corresponding boundedness result for the local fractional maximal function follows easily because for each $u\in L^{p}(\Omega)$, $p>1$, we have 
 \begin{equation}\label{bdd in Lp rn}
 \|\M_{\alpha,\Omega}u\|_{L^{p^*}(\Omega)}
 \le\|\M_{\alpha}(u\ch{\Omega})\|_{L^{p^*}(\rn)}
  \le C\|u\ch{\Omega}\|_{L^p(\rn)}
 = C\|u\|_{L^p(\Omega)}.
 \end{equation}
 
The local spherical fractional maximal function of $u$ is
\[
\mathcal S_{\alpha,\Omega}u(x) 
= \sup r^\alpha \vint{\partial B(x,r)}|u(y)|\,d\mathcal H^{n-1}(y),
\]
where the supremum is taken over all radii $r$ for which $0<r<\operatorname{dist}(x,\rn\setminus\Omega)$.
Observe that the barred integral denotes the integral average with respect to the Hausdorff measure $ \mathcal H^{n-1}$. 
When $\Omega=\rn$, the supremum is taken over all $r>0$ and we obtain the global spherical fractional maximal function 
$\mathcal S_{\alpha}u$.

The following norm estimate for the spherical fractional maximal operator will be useful for us.

\begin{theorem}\label{spherical}
Let $n\ge 2$, $p> n/(n-1)$ and $0\le \alpha<\min\{(n-1)/p,\ n-2n/((n-1)p)\}.$ Then
\begin{equation}
\|\mathcal S_\alpha u\|_{L^{p^*}(\rn)}\le C\|u\|_{L^p(\rn)},
\end{equation}
where $p^*=np/(n-\alpha p)$ and the constant $C$ depends only on $n$, $p$ and $\alpha$.
\end{theorem}
For $\alpha=0$, this was proved by Stein \cite{St} in the case $n\ge 3$ and by Bourgain \cite{Bo} in the case $n=2$. 
For $\alpha>0$, the result is due to Schlag \cite[Theorem 1.3]{S1} when $n=2$ and Schlag and Sogge \cite[Theorem 4.1]{SS} when $n\ge 3$. In \cite{S1} and \cite{SS} the result is stated for the operator
\[
\widetilde{\mathcal S} u(x)=\sup_{1<r<2}\,\vint{\partial B(x,r)}|u(y)|\,d\mathcal H^{n-1}(y),
\]
but the corresponding result for $\mathcal S_\alpha$ follows by the Littlewood-Paley theory as in 
  \cite[p.71--73]{Bo} , \cite[Section 2.4]{So} and \cite[Section 3.1]{S2}.
In particular, Theorem \ref{spherical} implies that the local spherical fractional maximal operator satisfies
\begin{equation}\label{local spherical}
\|\mathcal S_{\alpha,\Omega} u\|_{L^{p^*}(\Omega)}\le C\|u\|_{L^p(\Omega)}.
\end{equation}

\section{Derivative of the local fractional maximal function}\label{sec:rn}

In this section, we prove pointwise estimates for the gradient of the local fractional
maximal function. By integrating the pointwise estimates we also get the
corresponding norm estimates. 

We define the fractional average functions
$u_t^\alpha\colon\Omega\to[-\infty,\infty]$, $0<t<1$, $0 \leq \alpha < \infty$,
of a locally integrable function $u$ as
\begin{equation} \label{uta}
u_t^\alpha(x)
= (t\delta(x))^\alpha \vint{B(x,t\delta(x))}u(y)\,dy,
\end{equation}
where $\delta(x)=\operatorname{dist}(x, \rn\setminus\Omega)$.
We start by deriving an estimate for the gradient of the fractional average function
of an $L^p$-function.

\begin{lemma} \label{Duta lemma}
Let $p>n/(n-1)$, $0<t<1$ and
$1\leq\alpha < \min\{(n-1)/p,\ n-2n/((n-1)p)\}+1$.
If $u\in L^p(\Omega)$, then $|Du_t^\alpha| \in L^q(\Omega)$ with $q=np/(n-(\alpha-1) p)$. 
Moreover,
\begin{align} \label{Duta}
|Du_t^\alpha(x)|
\le C \big(\M_{\alpha-1,\Omega}u(x)+\mathcal S_{\alpha-1,\Omega}u(x) \big)
\end{align}
for almost every $x\in\Omega$, where the constant $C$ depends only on $n$.
\end{lemma}
\begin{proof}
Suppose first that $u\in L^p(\Omega) \cap C^\infty(\Omega)$.
According to Rademacher's theorem, as a Lipschitz function, $\delta$
is differentiable almost everywhere in $\Omega$.
Moreover, $|D\delta(x)|=1$ for almost every $x\in\Omega$.
Denoting $\omega_n=|B(0,1)|$, the Leibniz rule gives
\begin{align*}
D_iu_t^\alpha(x)
=&\:D_i\Big( \omega_n^{-1}(t\delta(x))^{\alpha-n} \Big)\int_{B(x,t\delta(x))}u(y)\,dy\\
&+\omega_n^{-1}(t\delta(x))^{\alpha-n}
D_i \left( \int_{B(x,t\delta(x))}u(y)\,dy \right), \qquad i=1,\dots,n,
\end{align*}
for almost every $x\in\Omega$, and by the chain rule
\begin{align*}
D_i \bigg( \int_{B(x,t\delta(x))}&u(y)\,dy \bigg)
=\int_{B(x,t\delta(x))}D_iu(y)\,dy\\
&+t D_i\delta(x) \int_{\partial B(x,t\delta(x))}u(y)\,d\mathcal H^{n-1}(y),
\qquad i=1,\dots,n,
\end{align*}
for almost every $x\in\Omega$.
Here we also used the fact that
$$
\frac\partial{\partial r}\int_{B(x,r)}u(y)\,dy
=\int_{\partial B(x,r)}u(y)\,d\mathcal H^{n-1}(y).
$$
Collecting the terms in a vector form, we obtain
\begin{equation} \label{Duta vector}
\begin{split}
Du_t^\alpha(x)
=\:&\omega_n^{-1}t^{\alpha-n}(\alpha-n)\delta(x)^{\alpha-n-1}D\delta(x)
\int_{B(x,t\delta(x))}u(y)\,dy \\
&+\omega_n^{-1}(t\delta(x))^{\alpha-n}
\int_{B(x,t\delta(x))}Du(y)\,dy \\
&+\omega_n^{-1}(t\delta(x))^{\alpha-n} t D\delta(x)
\int_{\partial B(x,t\delta(x))}u(y)\,d\mathcal H^{n-1}(y)
\end{split}
\end{equation}
for almost every $x\in\Omega$. Applying Gauss' theorem to the integral in the second term we have
$$
\int_{B(x,t\delta(x))}Du(y)\,dy
= \int_{\partial B(x,t\delta(x))}u(y)\nu(y)\,d\mathcal H^{n-1}(y),
$$
where $\nu(y)=(y-x)/(t\delta(x))$ is the unit outer normal of $B(x,t\delta(x))$.

Modifying the integrals into their average forms, we obtain
\begin{equation} \label{Duta average}
\begin{split}
Du_t^\alpha(x)
=\:&(\alpha-n)(t\delta(x))^\alpha\frac{D\delta(x)}{\delta(x)}
\vint{B(x,t\delta(x))}u(y)\,dy \\
&+n(t\delta(x))^{\alpha-1}
\vint{\partial B(x,t\delta(x))}u(y)\nu(y)\,d\mathcal H^{n-1}(y) \\
&+n(t\delta(x))^\alpha\frac{D\delta(x)}{\delta(x)}
\vint{\partial B(x,t\delta(x))}u(y)\,d\mathcal H^{n-1}(y)
\end{split}
\end{equation}
for almost every $x\in\Omega$. For the boundary integral terms, we have used the relation
between the Lebesgue measure of a ball and the Hausdorff measure of its boundary
$\mathcal H^{n-1}(\partial B(x,r)) = n\omega_nr^{n-1}$.

Taking the vector norms in the identity of the derivative and recalling that $0<t<1$
and $|D\delta(x)|=1$ for almost every $x\in\Omega$, we obtain
\begin{align*}
|Du_t^\alpha(x)|
\leq\:&|\alpha-n|(t\delta(x))^\alpha\frac{|D\delta(x)|}{\delta(x)}
\vint{B(x,t\delta(x))}|u(y)|\,dy \\
&+n(t\delta(x))^{\alpha-1}
\vint{\partial B(x,t\delta(x))}|u(y)||\nu(y)|\,d\mathcal H^{n-1}(y) \\
&+n(t\delta(x))^\alpha \frac{|D\delta(x)|}{\delta(x)}
\vint{\partial B(x,t\delta(x))}|u(y)|\,d\mathcal H^{n-1}(y) \\
\leq\:&n(t\delta(x))^{\alpha-1}
\vint{B(x,t\delta(x))}|u(y)|\,dy \\
&+n(t\delta(x))^{\alpha-1}
\vint{\partial B(x,t\delta(x))}|u(y)|\,d\mathcal H^{n-1}(y) \\
&+n(t\delta(x))^{\alpha-1}
\vint{\partial B(x,t\delta(x))}|u(y)|\,d\mathcal H^{n-1}(y) \\
\leq\:&C \big(\M_{\alpha-1,\Omega}u(x)+\mathcal S_{\alpha-1,\Omega}u(x) \big)
\end{align*}
for almost every $x\in\Omega$.
Thus, \eqref{Duta} holds for smooth functions.

The case $u\in L^p(\Omega)$ follows from an approximation argument.
For $u\in L^p(\Omega)$, there is a sequence $\{\varphi_j\}_j$ of functions
in $L^p(\Omega)\cap C^\infty(\Omega)$ such that $\varphi_j\to u$ in $L^p(\Omega)$ as $j\to\infty$.
Definition \eqref{uta} implies that
$$
u_t^\alpha(x)=\lim_{j\to\infty}(\varphi_j)_t^\alpha(x),
$$
when $x\in\Omega$. By the proved case for the smooth functions, we have
\begin{align} \label{Duta smooth}
\big|D(\varphi_j)_t^\alpha(x)\big|
\le C \big(\M_{\alpha-1,\Omega}\varphi_j(x)+\mathcal S_{\alpha-1,\Omega}\,\varphi_j(x) \big),
\qquad j=1,2,\dots,
\end{align}
for almost every $x\in\Omega$.
This inequality and the boundedness results \eqref{bdd in Lp rn} and \eqref{local spherical}
imply that
\begin{align*}
\|D(\varphi_j)_t^\alpha\|_{L^q(\Omega)}
&\le C\big(\|\M_{\alpha-1,\Omega}\varphi_j\|_{L^q(\Omega)}
+\|\mathcal S_{\alpha-1,\Omega}\,\varphi_j\|_{L^q(\Omega)}\big)
\\
&\le C\|\varphi_j\|_{L^p(\Omega)}, \qquad j=1,2,\dots,
\end{align*}
where $q=np/(n-(\alpha-1)p)$ and $C$ depends only on $n$, $p$ and $\alpha$.
Thus, $\{|D(\varphi_j)_t^\alpha|\}_j$ is a bounded
sequence in $L^q(\Omega)$ and has a weakly converging subsequence
$\{|D(\varphi_{j_k})_t^\alpha|\}_k$ in $L^q(\Omega)$.
Since $(\varphi_j)_t^\alpha$ converges pointwise to $u_t^\alpha$,
we conclude that the weak gradient $Du_t^\alpha$ exists
and that $|D(\varphi_{j_k})_t^\alpha|$ converges weakly to $|Du_t^\alpha|$ in
$L^q(\Omega)$ as $k\to\infty$. This follows from
the definitions of weak convergence and weak derivatives.

To establish \eqref{Duta}, we want to proceed to the limit in
\eqref{Duta smooth} as $j\to\infty$.
By the sublinearity of the maximal operator and
\eqref{bdd in Lp rn}, we obtain
\begin{align*}
\|\M_{\alpha-1,\Omega}\varphi_j-\M_{\alpha-1,\Omega}u\|_{L^q(\Omega)}
&\le\|\M_{\alpha-1,\Omega}(\varphi_j-u)\|_{L^q(\Omega)}
\\
&\le C\|\varphi_j-u\|_{L^p(\Omega)}, \qquad j=1,2,\dots.
\end{align*}
Analogously, by \eqref{local spherical}, we get
\begin{align*}
\|\mathcal S_{\alpha-1,\Omega}\,\varphi_j-\mathcal S_{\alpha-1,\Omega}u\|_{L^q(\Omega)}
\le C\|\varphi_j-u\|_{L^p(\Omega)}, \qquad j=1,2,\dots.
\end{align*}
Hence $\M_{\alpha-1,\Omega}\varphi_j+\mathcal S_{\alpha-1,\Omega}\,\varphi_j$ converges
to $\M_{\alpha-1,\Omega}u+\mathcal S_{\alpha-1,\Omega}u$ in $L^q(\Omega)$ as $j \to \infty$.

To complete the proof, we need the following simple property of weak convergence:
If $f_k\to f$ and $g_k\to g$ weakly in $L^q(\Omega)$ and $f_k\le g_k$, $k=1,2,\dots$, almost everywhere in $\Omega$, then 
$f\le g$ almost everywhere in $\Omega$.
Applying the property to \eqref{Duta smooth} with
$$
f_k=\big|D(\varphi_{j_k})_t^\alpha\big| \quad \text{and} \quad
g_k=C\big(\M_{\alpha-1,\Omega}\varphi_{j_k}+\mathcal S_{\alpha-1,\Omega}\,\varphi_{j_k}\big),
$$
we obtain \eqref{Duta}.
This completes the proof.
\end{proof}

The gradient of the local fractional maximal function of an $L^{p}$-function satisfies a pointwise estimate in terms
of a local fractional maximal function and local spherical fractional maximal function of the function itself.
The following is the main result of this section.
\begin{theorem} \label{DMaOu theorem}
Let $p>n/(n-1)$ and let 
$1\leq\alpha<\min\{(n-1)/p,\ n-2n/((n-1)p)\}+1$.
If $u\in L^p(\Omega)$, then $|D\M_{\alpha,\Omega}u| \in L^q(\Omega)$ with $q=np/(n-(\alpha-1)p)$. 
Moreover,
\begin{equation} \label{DMaOu}
|D\M_{\alpha,\Omega}u(x)|\le C \big(\M_{\alpha-1,\Omega}u(x)+\mathcal S_{\alpha-1,\Omega}u(x)\big)
\end{equation}
for almost every $x\in\Omega$, where the constant $C$ depends only on $n$.
\end{theorem}
\begin{proof}
Let $t_j$, $j=1,2,\dots$, be an enumeration of the rationals
between $0$ and $1$ and let 
\[
u_j=|u|_{t_j}^\alpha, \qquad j=1,2,\dots.
\]
By Lemma \ref{Duta lemma}, we see that $|Du_j|\in L^q(\Omega)$
for every $j=1,2,\dots$ and \eqref{Duta} gives us the estimate
$$
|Du_j(x)|
\le  C \big(\M_{\alpha-1,\Omega}u(x)+\mathcal S_{\alpha-1,\Omega}u(x)\big),
\qquad j=1,2,\dots,
$$
for almost every $x\in\Omega$.
We define $v_k\colon\Omega\to[-\infty,\infty]$ as the pointwise maximum
$$
v_k(x)
=\max_{1\le j\le k}u_j(x),
\qquad k=1,2,\dots.
$$
Then $\{v_k\}_k$ is an increasing sequence of functions converging pointwise to 
$\M_{\alpha,\Omega} u$.
Moreover, the weak gradients $Dv_k$, $k=1,2,\dots$, exist since
$Du_j$ exists for each $j=1,2,\dots$, and we can estimate
\begin{equation} \label{Dvkx}
\begin{split}
|Dv_k(x)|
&=\big|D\max_{1\le j\le k}u_j(x)\big|
\le\max_{1\le j\le k}|Du_j(x)|\\
& \le C \big(\M_{\alpha-1,\Omega}u(x)+\mathcal S_{\alpha-1,\Omega}u(x)\big),
\qquad k=1,2,\dots,
\end{split}
\end{equation}
for almost every $x\in\Omega$.

The rest of the proof goes along the lines of the final part of the proof for Lemma \ref{Duta lemma}.
By \eqref{Dvkx}, \eqref{bdd in Lp rn} and \eqref{local spherical},
we obtain
\begin{align*}
\|Dv_k\|_{L^q(\Omega)}
&\le C\big(\|\M_{\alpha-1,\Omega}u\|_{L^q(\Omega)}
+\|\mathcal S_{\alpha-1,\Omega}u\|_{L^q(\Omega)}\big)\\
&\le C\|u\|_{L^p(\Omega)}, \qquad k=1,2,\dots.
\end{align*}
Hence $\{|Dv_k|\}_k$ is a bounded sequence in $L^q(\Omega)$
with $v_k\to\M_{\alpha,\Omega} u$ pointwise in $\Omega$ as $k\to\infty$. 
Thus, there is a weakly converging subsequence $\{|Dv_{k_j}|\}_j$
that has to converge weakly to $|D\M_{\alpha,\Omega}u|$ in $L^q(\Omega)$ as $j\to\infty$.
We may proceed to the weak limit in \eqref{Dvkx}, using the
same argument as in the end of the proof of Lemma \ref{Duta lemma},
and claim \eqref{DMaOu} follows.
\end{proof}

\begin{corollary} \label{bounded sobolev}
Let $p>n/(n-1)$ and let $1\leq\alpha<n/p$.
If $|\Omega|<\infty$  and $u \in L^p(\Omega)$,
then $\mathcal M_{\alpha,\Omega}u \in W^{1,q}(\Omega)$ with $q=np/(n-(\alpha-1)p)$.
\end{corollary}
\begin{proof}
By \eqref{bdd in Lp rn} we have $\mathcal M_{\alpha,\Omega}u \in L^{p^*}(\Omega)$ 
and $|D \mathcal M_{\alpha,\Omega}u| \in L^q(\Omega)$ by Theorem \ref{DMaOu theorem} because
\[
\frac np \le \min\left\{\frac{n-1}{p},\ n-\frac{2n}{(n-1)p}\right\}+1.
\]
Since $q<p^*$, we have
\[
\|\M_{\alpha,\Omega}u\|_{L^q(\Omega)}
\le |\Omega|^{1/q-1/p^*} \|\M_{\alpha,\Omega}u\|_{L^{p^*}(\Omega)} < \infty
\]
by H\"older's inequality. Hence $\M_{\alpha,\Omega}u \in W^{1,q}(\Omega)$.
\end{proof}
Next we will show that the local fractional maximal operator actually maps $L^{p}(\Omega)$ to the Sobolev space with zero boundary values.  
For this we need the following Hardy-type result proved in \cite[Theorem 3.13]{KiMa}.

\begin{theorem}\label{hardy sobo nolla rn}
Let $\Omega\subset \rn$, $\Omega\ne\rn$, be an open set. If $u\in W^{1,p}(\Omega)$ and 
 \[
 \int_{\Omega}\bigg(
 \frac{u(x)}{\dist(x,\rn\setminus\Omega)}\bigg)^{p}\,dx<\infty,
\]
then $u\in W^{1,p}_{0}(\Omega)$.
\end{theorem}

\begin{corollary}\label{sobo nolla rn}
Let $\Omega\subset\rn$ be an open set with $|\Omega|<\infty$. Let $p>n/(n-1)$ and $1\le \alpha<n/p$. 
If $u \in L^p(\Omega)$, then $\mathcal M_{\alpha,\Omega}u \in W^{1,q}_0(\Omega)$ with $q=np/(n-(\alpha-1)p)$.
\end{corollary}

\begin{proof}
By Corollary \ref{bounded sobolev}, $\M_{\alpha,\Omega}u\in W^{1,q}(\Omega)$. It suffices to show that 
\begin{equation}\label{hardy rn}
 \int_{\Omega}\bigg(
 \frac{\M_{\alpha,\Omega}u(x)}{\dist(x,\rn\setminus\Omega)}\bigg)^{q}\,dx<\infty.
\end{equation}
The claim then follows from Theorem \ref{hardy sobo nolla rn}.
Since $$\M_{\alpha,\Omega}u(x)\le\dist(x,\rn\setminus\Omega)\M_{\alpha-1,\Omega}u(x)$$ for every $x\in\Omega$, inequality 
\eqref{hardy rn} follows from \eqref{bdd in Lp rn}. Hence $\M_{\alpha,\Omega}u\in W_0^{1,q}(\Omega)$.
\end{proof}

Next we derive estimates for Sobolev functions.
In general, Sobolev functions do satisfy neither any better inequality for gradients nor better embedding than $L^{p}$-functions,
but since no spherical maximal function is needed in the Sobolev setting, the estimate holds also when $1<p\le n/(n-1)$.
The following is a variant of Lemma \ref{Duta lemma}.

\begin{lemma} \label{Duta lemma sobo}
Let $1<p<n$, $1\leq\alpha<n/p$ and let $0<t<1$. 
If $|\Omega|<\infty$ and $u\in W^{1,p}(\Omega)$, then 
$|Du_t^\alpha| \in L^q(\Omega)$ with $q=np/(n-(\alpha-1)p)$. Moreover,
\begin{equation} \label{Duta sobo}
|Du_t^\alpha(x)|\le 2 \M_{\alpha,\Omega}|Du|(x) + \alpha \M_{\alpha-1,\Omega} u(x)
\end{equation}
for almost every $x\in\Omega$.
\end{lemma}
\begin{proof}
Suppose first that $u\in W^{1,p}(\Omega) \cap C^\infty(\Omega)$.
Equation \eqref{Duta vector} in the proof of Lemma \ref{Duta lemma} holds in this case, as well,
and modifying the integrals into average forms we obtain
\begin{equation*}
\begin{split}
Du_t^\alpha(x)
=\:&\alpha(t\delta(x))^\alpha\frac{D\delta(x)}{\delta(x)}
\vint{B(x,t\delta(x))}u(y)\,dy \\
&+n(t\delta(x))^\alpha\frac{D\delta(x)}{\delta(x)}
\left( \vint{\partial B(x,t\delta(x))}u(y)\,d\mathcal H^{n-1}(y) - \vint{B(x,t\delta(x))}u(y)\,dy \right)\\
&+(t\delta(x))^\alpha
\vint{B(x,t\delta(x))} Du(y)\,dy
\end{split}
\end{equation*}
for almost every $x\in\Omega$.

In order to estimate the difference of the two integrals
in the parenthesis, we use Green's first identity
$$
\int_{\partial B(x,r)}u(y)
\frac{\partial v}{\partial\nu}(y)\,d\mathcal H^{n-1}(y)
=\int_{B(x,r)}
\big(u(y)\Delta v(y)+D u(y)\cdot Dv(y)\big)\,dy,
$$
where $\nu(y)=(y-x)/r$ is the unit outer normal of $B(x,r)$.
We choose $r=t\delta(x)$ and $v(y)=|y-x|^2/2$.
With these choices 
\[
Dv(y)=y-x, \quad\frac{\partial v}{\partial\nu}(y)=r, \quad\Delta v(y)=n
\] 
and Green's formula reads
$$
\vint{\partial B(x,t\delta(x))}u(y)\,d\mathcal H^{n-1}(y)
-\vint{B(x,t\delta(x))}u(y)\,dy
=\frac1n\,\vint{B(x,t\delta(x))}Du(y)\cdot(y-x)\,dy.
$$

Taking the vector norms in the identity of the derivative and recalling
that $|D\delta(x)|=1$ almost everywhere and $0<t<1$, we obtain
\begin{align*}
|Du_t^\alpha(x)|
\leq\:&\alpha(t\delta(x))^\alpha\frac{|D\delta(x)|}{\delta(x)}
\vint{B(x,t\delta(x))}|u(y)|\,dy \\
&+n(t\delta(x))^\alpha\frac{|D\delta(x)|}{\delta(x)}
\frac1n\vint{B(x,t\delta(x))}|Du(y)||y-x|\,dy \\
&+(t\delta(x))^\alpha
\vint{B(x,t\delta(x))}|Du(y)|\,dy \\
\leq\:&\alpha (t\delta(x))^{\alpha-1}
\vint{B(x,t\delta(x))}|u(y)|\,dy \\
&+(t\delta(x))^\alpha
\vint{B(x,t\delta(x))}|Du(y)|\,dy \\
&+(t\delta(x))^\alpha
\vint{B(x,t\delta(x))}|Du(y)|\,dy \\
\leq\:&\alpha \M_{\alpha-1,\Omega}u(x)+2\M_{\alpha,\Omega}|Du|(x)
\end{align*}
for almost every $x\in\Omega$. Thus, \eqref{Duta sobo} holds for smooth functions.

The case $u\in W^{1,p}(\Omega)$
follows from an approximation argument.
For $u\in W^{1,p}(\Omega)$, there is a sequence $\{\varphi_j\}_j$
of functions in $W^{1,p}(\Omega)\cap C^\infty(\Omega)$ such that
$\varphi_j\to u$ in $W^{1,p}(\Omega)$ as $j\to\infty$.
By definition \eqref{uta} we see that
$$
u_t^\alpha(x)=\lim_{j\to\infty}(\varphi_j)_t^\alpha(x),
$$
when $x\in\Omega$.
By the proved case for smooth functions we have
\begin{align} \label{Duta smooth sobo}
\big|D(\varphi_j)_t^\alpha(x)\big|
\le 2 \M_{\alpha,\Omega}|D\varphi_j|(x) + \alpha \M_{\alpha-1,\Omega}\varphi_j(x),
\qquad j=1,2,\dots,
\end{align}
for almost every $x\in\Omega$. Let $p^*=np/(n-\alpha p)$ and $q=np/(n-(\alpha-1)p)$.
Then $\|f\|_{L^q(\Omega)}<C\|f\|_{L^{p^*}(\Omega)}$ for any $f \in L^{p^*}(\Omega)$
since $q<p^*$ and $|\Omega|<\infty$.
The estimate \eqref{Duta smooth sobo} and the boundedness
result \eqref{bdd in Lp rn} imply
\begin{align*}
\big\|D(\varphi_j)_t^\alpha\big\|_{L^q(\Omega)}
&\le 2\big\|\M_{\alpha,\Omega}|D\varphi_j|\big\|_{L^q(\Omega)}
+\alpha\big\|\M_{\alpha-1,\Omega}\varphi_j\big\|_{L^q(\Omega)}
\\
&\le C\big\|\M_{\alpha,\Omega}|D\varphi_j|\big\|_{L^{p^*}(\Omega)}
+\alpha\big\|\M_{\alpha-1,\Omega}\varphi_j\big\|_{L^q(\Omega)}
\\
&\le C \|D\varphi_j\|_{L^p(\Omega)}
+C\|\varphi_j\|_{L^p(\Omega)}
\\
&\le C\|\varphi_j\|_{W^{1,p}(\Omega)}, \qquad j=1,2,\dots,
\end{align*}
where $C$ depends on $n$, $p$, $\alpha$ and $|\Omega|$.
Thus, $\{D(\varphi_j)_t^\alpha\}_j$ is a bounded
sequence in $L^q(\Omega)$ and has a weakly converging subsequence
$\{D(\varphi_{j_k})_t^\alpha\}_k$. Since $(\varphi_j)_t^\alpha$ converges to $u_t^\alpha$
pointwise, we conclude that the Sobolev derivative $Du_t^\alpha$ exists
and that $D(\varphi_{j_k})_t^\alpha\to Du_t^\alpha$ weakly in
$L^q(\Omega)$ as $k\to\infty$.

To establish \eqref{Duta sobo}, we want to proceed to the limit in
\eqref{Duta smooth sobo} as $j\to\infty$.
This goes as in the proof of Lemma \ref{Duta lemma},
and we obtain the claim.
\end{proof}

The following is a variant of Theorem \ref{DMaOu theorem} for Sobolev functions.

\begin{theorem} \label{DMaOu theorem sobo}
Let $1<p<n$ and let $1\leq\alpha<n/p$. If $|\Omega|<\infty$ and $u\in W^{1,p}(\Omega)$, 
then $\M_{\alpha,\Omega}u \in W^{1,q}(\Omega)$ with $q=np/(n-(\alpha-1)p)$. Moreover,
\[
|D\M_{\alpha,\Omega}u(x)|\le 2 \M_{\alpha,\Omega}|Du|(x) + \alpha \M_{\alpha-1,\Omega} u(x)
\]
for almost every $x\in\Omega$.
\end{theorem}
The proof is analogous to the proofs of Theorem \ref{DMaOu theorem} and Corollary \ref{bounded sobolev},
but using Lemma \ref{Duta lemma sobo} instead of Lemma \ref{Duta lemma}.

\begin{remark}
If $\Omega$ is bounded with a $C^1$-boundary, then Theorem \ref{DMaOu theorem sobo}
holds with a better exponent
$p^*=np/(n-\alpha p)$ instead of $q$. Indeed, in this setting we have the Sobolev inequality
\[
\|u\|_{L^r(\Omega)} \le C \|u\|_{W^{1,p}(\Omega)},
\]
where $r=np/(n-p)$ is the Sobolev conjugate of $p$, and we can estimate
\begin{align*}
\|D\M_{\alpha,\Omega}u\|_{L^{p^*}(\Omega)} &\le 2\|\M_{\alpha,\Omega}|Du|\|_{L^{p^*}(\Omega)}+\alpha\|\M_{\alpha-1,\Omega}u\|_{L^{p^*}(\Omega)} \\
&\le C\|Du\|_{L^p(\Omega)}+C\|u\|_{L^r(\Omega)} \\
&\le C\|Du\|_{L^p(\Omega)}+C\|u\|_{W^{1,p}(\Omega)} \\
&\le C\|u\|_{W^{1,p}(\Omega)}.
\end{align*}
In the second inequality, we used \eqref{bdd in Lp rn}
and the fact that $p^*$ can be written as
$p^*=nr/(n-(\alpha-1)r)$.
\end{remark}

\section{Examples}\label{sec: examples}
\medskip

Our first example shows that the inequality 
\begin{equation} \label{CMa-1Oux}
|D \mathcal M_{\alpha,\Omega}u(x)| \le C \mathcal M_{\alpha-1,\Omega}u(x),
\end{equation}
for almost every $x \in \Omega$, cannot hold in general. 
Hence, the term containing the spherical maximal function in \eqref{DMaOu} cannot be dismissed. 

\begin{example}\label{Jannen esim}

Let $\Omega = B(0,1) \subset \rn$. Let $1<p<\infty$ and let $0<\beta<1$. Then the function $u$, 
\[
u(x)=(1-|x|)^{-\beta/p},
\]
belongs to $L^p(\Omega) \cap L^1(\Omega)$. When $0<|x|<\rho$, $\rho$ small enough, the maximizing radius for the maximal functions
$\mathcal M_{\alpha,\Omega}u(x)$ and $\mathcal M_{\alpha-1,\Omega}u(x)$, 
$\alpha \ge 1$, is clearly the largest possible, i.e. $1-|x|$. 
Thus, by \eqref{Duta average} in the proof of Lemma \ref{Duta lemma},
\begin{equation*}
\begin{split}
D\mathcal M_{\alpha,\Omega}u(x)
=\:&(n-\alpha)\frac{x}{|x|}(1-|x|)^{\alpha-1}
\vint{B(x,1-|x|)}u(y)\,dy \\
&+n(1-|x|)^{\alpha-1}
\vint{\partial B(x,1-|x|)}u(y)\nu(y)\,d\mathcal H^{n-1}(y) \\
&-n\frac{x}{|x|}(1-|x|)^{\alpha-1}
\vint{\partial B(x,1-|x|)}u(y)\,d\mathcal H^{n-1}(y).
\end{split}
\end{equation*}
By symmetry, the contribution from the integral in the second term has the same direction $\frac{x}{|x|}$ as the first term,
whereas the direction of the last term is the opposite. Thus, all the terms lie in the same line of $\rn$
and it is sufficient to compare the vector norm of the first term and the sum of the latter terms. For the first term,
\[
\bigg|(n-\alpha)\frac{x}{|x|}(1-|x|)^{\alpha-1} \vint{B(x,1-|x|)}u(y)\,dy\bigg|
= (n-\alpha) \mathcal M_{\alpha-1,\Omega}u(x) \le M,
\]
where $M$ depends only on $n$, $p$, $\alpha$, $\beta$ and $\rho$. For the latter terms,
\[
\bigg|\vint{\partial B(x,1-|x|)}u(y)\big(\nu(y)-\frac{x}{|x|}\big)\,d\mathcal H^{n-1}(y)\bigg|
\ge \frac1{2} \vint{S(x)} u(y)\,d\mathcal H^{n-1}(y),
\]
where $S(x)$ is the half sphere $S(x)=\{y \in \partial B(x,1-|x|) : (y-x) \cdot x < 0\}$.
Further, when $|x|<\eps$,
\[
n(1-|x|)^{\alpha-1}\frac1{2} \vint{S(x)} u(y)\,d\mathcal H^{n-1}(y) 
\geq \frac{n(1-\eps)^{\alpha-1}}{2(2\eps)^{\beta/p}},
\]
which goes to $\infty$ as $\eps \to 0$. We conlude that for small values of $|x|$,
the boundary integral terms dominate, and thus \eqref{CMa-1Oux} cannot hold.
\end{example}

The next  example shows that Theorem \ref{DMaOu theorem sobo} is sharp.
There are domains $\Omega\subset\rn$, $n\ge 2$, for which 
$\M_{\alpha,\Omega}(W^{1,p}(\Omega))\not\subset W^{1,r}(\Omega)$ when $r>q=np/(n-(\alpha-1)p)$.
This is in strict contrast with the global case, where 
$\M_{\alpha}\colon W^{1,p}(\rn)\hookrightarrow W^{1,p^*}(\rn)$ with $p^*=np/(n-\alpha p)$, see \cite[Theorem 2.1]{KS}. 

\begin{example}\label{example sobo}
Let 
\[
\Omega=\operatorname{int}\Big(\bigcup_{k=1}^\infty B_k\cup C_k\Big),
\] 
where 
\[
B_k=[k,k+2^{-k}]\times [0,2^{-k}]^{n-1}
\quad\text{and}\quad 
C_k=[k+2^{-k},k+1]\times [0,2^{-3k}]^{n-1}
\] 
is a corridor connecting $B_k$ to $B_{k+1}$. 
It suffices to show that for every $p'>p,$ there exists $u\in W^{1,p}(\Omega)$ such that 
\[
|D\M_{\alpha,\Omega}u|\not\in L^{np'/(n-(\alpha-1)p')}(\Omega).
\] 
Let $p'>p$. Define $u$ such that $u=2^{kn/p'}$ on $B_k$ and $u$ increases linearly from $2^{kn/p'}$
to $2^{(k+1)n/p'}$ on $C_k$. Then it is easy to see that $u\in W^{1,p}(\Omega)$.

If $x\in\frac12 B_k$, where $\frac12 B_k$ is a cube with the same center as $B_k$ and with side length half side length of $B_k$,
we have that 
\[
\M_{\alpha,\Omega} u(x)=\dist(x,\rn\setminus B_k)^\alpha 2^{kn/p'}.
\]
Hence, for almost every $x\in\frac12 B_k$,
\[
|D\M_{\alpha,\Omega} u(x)|
= \alpha \dist(x,\rn\setminus B_k)^{\alpha-1}2^{kn/p'}
\ge C2^{-k(\alpha-1-n/p')}, 
\]
which implies that
\[
\int_{\Omega}|D\M_{\alpha,\Omega} u(x)|^{np'/(n-(\alpha-1)p')}\,dx
\ge C\sum_{k=1}^\infty\int_{\frac12 B_k}2^{nk}\,dx=\infty.
\]
\end{example}





\bigskip
Define the local fractional maximal function over cubes by setting
\[
\widetilde \M_{\alpha,\Omega}u(x)
=\sup_{Q(x,r)\subset\Omega}\, r^{\alpha} \vint{Q(x,r)}|u(y)|\,dy,
\]
where $Q(x,r)=(x_1-r,x_1+r)\times\dots\times(x_n-r,x_n+r)$ is a cube with center $x=(x_1,\dots,x_n)$ and of side length $2r$.
As noted in \cite{KS}, in the global case the maximal operator over cubes behaves similarly as the maximal operator over balls.
Somewhat surprisingly, in the local case, the smoothing properties of the maximal operator over cubes are much worse.
Indeed, we show that there are domains $\Omega\subset\rn$ such that 
$\M_{\alpha,\Omega}(L^{p}(\Omega))\not\subset W^{1,p'}(\Omega)$ when $p'>p$.

\begin{example}
Let $\Omega=(0,2)\times(-1,2)^{n-1}$ and let $u\colon\Omega\to\mathbb{R}$ be of the form $u(x)=v(x_1),$ where $v$ is non-negative and continuous. If $Q(x,r)\subset\Omega,$ then
\[
r^{\alpha} \vint{Q(x,r)}|u(y)|\,dy=\frac12 r^{\alpha-1}\int_{x_1-r}^{x_1+r}v(t)\,dt.
\]
Hence, for $\alpha> 1$ and $x\in(0,1)^n$, we have
\[
\widetilde \M_{\alpha,\Omega}u(x)=\frac12 x_1^{\alpha-1}\int_0^{2x_1}v(t)\,dt
\]
and
\[
D_1 \widetilde\M_{\alpha,\Omega} u(x)
=\frac12(\alpha-1)x_1^{\alpha-2}\int_0^{2x_1}v(t)dt \ + \ x_1^{\alpha-1}v(2x_1).
\]
It follows that
\[
D_1 \widetilde\M_{\alpha,\Omega} u(x)\ge Cv(2x_1),
\]
for $x\in(1/2,1)\times(0,1)^{n-1},$ which shows that $D_1 \widetilde\M_{\alpha,\Omega} u$ cannot belong to a higher $L^p$ space than $u$.
\end{example}


In all our results in Section \ref{sec:rn}, we assumed that $\alpha\ge 1$. Our final example shows that, in the case $0<\alpha<1$,  
$\M_{\alpha,\Omega}u$ can be very irregular, even when $u$ is constant function. Indeed, we show that for any $r>0$, there exists a domain $\Omega$ such that the gradient of the fractional maximal function of a constant function does not belong to $L^r(\Omega)$.


\begin{example}
Let $n\ge 1$, $0<\alpha<1$ and $r>0$. 
We will construct a bounded open set $\Omega\subset \mathbb{R}^n$ such that, for $u\equiv 1$, we have 
\[
\M_{\alpha,\Omega} u=\dist(\cdot,\rn\setminus\Omega)^\alpha
\]
and the gradient of $\M_{\alpha,\Omega}u$ does not belong to $L^{r}(\Omega)$. 
Let $\beta$ be an integer satisfying $\beta\ge n/((1-\alpha)r)$, and let 
\[
\Omega=B(0,2)\setminus \overline{\bigcup_{k\ge 1} S_k},
\]
where
\[
S_k=\{2^{-k}+j2^{-(1+\beta) k}:j=1,\dots,2^{\beta k}\}^n.
\] 
If $x\in S_k$ and $y\in S_l$ with $x\neq y$, then the balls $B(x,2^{-(1+\beta)k-1})$ and $B(y,2^{-(1+\beta)l-1})$ are disjoint. 
For each $y\in B(x,2^{-(1+\beta)k-1})\setminus\{x\}$, we have $\M_{\alpha,\Omega} u(y)=|y-x|^\alpha$, which implies that 
\[
|D \M_{\alpha,\Omega} u(y)|
=\alpha |y-x|^{\alpha-1}\ge C 2^{-(1+\beta)(\alpha-1)k}.
\] 
It follows that
\[
\begin{split}
\int_\Omega |D \M_{\alpha,\Omega} u(y)|^r\,dy
&\ge \sum_{k\ge 1}\sum_{x\in S_k}
\int_{B(x,2^{-(1+\beta)k-1})}|D \M_{\alpha,\Omega} u(y)|^r\,dy\\
&\ge C\sum_{k\ge 1}2^{\beta k n}2^{-(1+\beta)kn}2^{-(1+\beta)(\alpha-1)rk}\\
&= C\sum_{k\ge 1}2^{((1+\beta)(1-\alpha)r-n)k}=\infty,
\end{split}
\]
and hence the gradient of $\M_{\alpha,\Omega}u$ does not belong to $L^{r}(\Omega)$.
\end{example}

\section{The local discrete fractional maximal function in metric space}\label{sec: metric}
In this section, we study the smoothing properties of the local discrete fractional maximal function in a metric space which is equipped with a doubling measure.
We begin by recalling some definitions.
 
\subsection{Sobolev spaces on metric spaces}
Let $X=(X, d,\mu)$ be a metric measure space equipped with a metric $d$ and a Borel regular, doubling outer
measure $\mu$. The doubling property means that there is a fixed constant
$c_d>0$, called a doubling constant of $\mu$, such that 
\[
\mu(B(x,2r))\le c_d\mu(B(x,r))
\]
for each ball $B(x,r)=\{y\in X: d(y,x)<r\}$. 
We also assume that open sets have positive and bounded sets finite measure.
We say that the measure $\mu$ satisfies a measure lower bound condition if there exist constants $Q\ge 1$ and $c_{l}>0$ such that
\begin{equation}\label{measure lower bound}
 \mu(B(x,r))\ge c_{l}r^{Q}
\end{equation}
for all $x\in X$ and $r>0$. This assumption is needed for the boundedness of the fractional maximal operator in $L^{p}$.

General metric spaces lack the notion of smooth functions, but there exists a natural counterpart of Sobolev spaces, defined by 
Shanmugalingam in \cite{Sh1} and based on upper gradients.
A Borel function $g\ge0$ is an upper gradient of a function $u$ on an open set $\Omega\subset X$, if for all curves
$\gamma$ joining points $x$ and $y$ in $\Omega$,
\begin{equation}\label{ug}
|u(x)-u(y)|\le\int_{\gamma}g\,ds,
\end{equation}
whenever both $u(x)$ and $u(y)$ are finite, and $\int_{\gamma}g\,ds=\infty$ otherwise. 
By a curve, we mean a nonconstant, rectifiable, continuous mapping from a compact interval to $X$. 

If $g\ge0$ is a measurable function and \eqref{ug} only fails for a curve family with zero $p$-modulus, then $g$ is a $p$-weak upper gradient of $u$ on $\Omega$. For the $p$-modulus on metric measure spaces and the properties of upper gradients, see for example \cite{BB}, \cite{Hj2}, \cite{HK}, \cite{Sh1}, and \cite{Sh2}.
If $1\le p<\infty$ and $u\in L^{p}(\Omega)$, let
\[
\|u\|_{N^{1,p}(\Omega)} 
= \bigg(\int_\Omega |u|^p\,d\mu 
+\inf_g\int_\Omega g^p \,d\mu\bigg)^{1/p},
\]
where the infimum is taken over all $p$-weak upper gradients of $u$.
The Sobolev space on $\Omega$ is the quotient space
\[
N^{1,p}(\Omega)=\{u:\|u\|_{N^{1,p}(\Omega)}<\infty\}/{\sim},
\]
where  $u \sim v$ if and only if $\|u-v\|_{N^{1,p}(\Omega)}=0$.

For a measurable set $E\subset X$, the Sobolev space with zero boundary values is 
\[
N^{1,p}_0(E)=\bigl\{u|_E:u\in N^{1,p}(X)\text{ and }u=0
                \text{ in }X\setminus E\bigr\}.
\]
By \cite[Theorem 4.4]{Sh2}, also the space $N^{1,p}_0(E)$, equipped with the norm inherited from $N^{1,p}(X)$, is a Banach space. 
Note that we obtain the same class of functions as above if we require $u$ to vanish $p$-quasi everywhere in $X\setminus E$ in the sense of $p$-capacity, since Sobolev functions are defined pointwise outside sets of zero capacity, see \cite{Sh1} and \cite{BBS}.

In Theorems \ref{hardy sobo nolla} and \ref{sobo nolla}, we assume, in addition to the doubling condition, that $X$ supports a (weak) $(1,p)$-Poincar\'e inequality, which means that there exist constants $c_P>0$ and $\lambda\ge1$ such that for all balls $B$, all locally integrable functions $u$ and for all $p$-weak upper gradients $g_u$ of $u$, we have 
\[
\vint{B}|u-u_B|\, d\mu
\le c_Pr\bigg(\vint{\lambda B}g_u^p\,d\mu\bigg)^{1/p},
\]
where 
\[
 u_B=\vint{B}u\,d\mu={\mu(B)}^{-1}\int_Bu\,d\mu
\]
is the integral average of $u$ over $B$.

In the Euclidean space with the Lebesgue measure, $N^{1,p}(\Omega)=W^{1,p}(\Omega)$ for all domains $\Omega\subset\rn$ and
$g_u=|Du|$ is a minimal upper gradient of $u$, see \cite{Sh1} and \cite{Sh2}.
Standard examples of doubling metric spaces supporting Poincar\'e inequalities include (weighted) Euclidean spaces, compact Riemannian manifolds, metric graphs, and Carnot-Carath\'eodory spaces. See for instance \cite{HjK} and \cite{Hj2}, and the references therein, for more extensive lists of examples and applications.

The following Hardy-type condition for functions in Sobolev spaces with zero boundary values has been proved in \cite{A} and in \cite{KKM}.
\begin{theorem}\label{hardy sobo nolla}
Assume that $X$ supports a $(1,p)$-Poincar\'e inequality with $1<p<\infty$. 
Let $\Omega\subset X$ be an open set. If $u\in N^{1,p}(\Omega)$ and 
\[
\int_{\Omega}\bigg(
\frac{u(x)}{\dist(x,X\setminus\Omega)}\bigg)^{p}\,d\mu(x)<\infty,
\]
then $u\in N^{1,p}_{0}(\Omega)$.
\end{theorem}

\subsection{The fractional maximal function}

Let $\Omega\subset X$ be an open set such that $X\setminus\Omega\ne\emptyset$ and let $\alpha\ge 0$. 
The local fractional maximal function of a locally integrable function $u$ is 
\[
\M_{\alpha,\Omega} u(x)=\sup\,r^{\alpha}\vint{B(x,r)}|u|\,d\mu,
\]
where the supremum is taken over all radii $r$ satisfying $0<r<\dist(x,X\setminus\Omega)$.
If $\alpha=0$, we have the local Hardy-Littlewood maximal function 
\[
\M_{\Omega} u(x)=\sup\vint{B(x,r)}|u|\,d\mu. 
\]
When $\Omega=X$, the supremum is taken over all $r>0$ and we obtain the fractional maximal function 
$\M_{\alpha}u$ and the Hardy-Littlewood maximal function $\M u$.

Sobolev type theorem for the fractional maximal operator follows easily from  the Hardy-Littlewood maximal function theorem. 
For the proof, see \cite{EKM}, \cite{GGKK} or \cite{HKNT}.

\begin{theorem}\label{fracM bounded}
Assume that measure lower bound condition \eqref{measure lower bound} holds. 
If $p>1$ and $0<\alpha<Q/p$, then there is a constant $C>0$, independent of $u$,
such that
 \[
 \|\M_{\alpha}u\|_{L^{p^*}(X)}
 \le C\|u\|_{L^p(X)},
 \]
for every $u\in L^{p}(X)$ with $p^*=Qp/(Q-\alpha p)$. 
If $p=1$ and $0<\alpha<Q$, then 
 \[
 \mu(\{\M_{\alpha}u>\lambda\})
 \le C\big(\lambda^{-1}\|u\|_{L^{1}(X)}\big)^{Q/(Q-\alpha)}
 \]
for every $u\in L^{1}(X)$. The constant $C>0$ depends only on the doubling constant, the constant in the measure lower bound and $\alpha$. 
\end{theorem}
Now the corresponding boundedness results for the local fractional maximal function follow easily because for each open set 
$\Omega\subset X$ and for each $u\in L^{p}(\Omega)$, $p>1$, we have 
 \begin{equation}\label{bdd in Lp}
 \|\M_{\alpha,\Omega}u\|_{L^{p^*}(\Omega)}
 \le\|\M_{\alpha}(u\ch{\Omega})\|_{L^{p^*}(X)}
  \le C\|u\ch{\Omega}\|_{L^p(X)}
 = C\|u\|_{L^p(\Omega)}.
 \end{equation}
 Similarly, we obtain a weak type estimate when $p=1$,
\begin{equation}\label{weak type in L1}
 \mu(\{x\in\Omega:\M_{\alpha,\Omega}u(x)>\lambda\})
 \le C\big(\lambda^{-1}\|u\|_{L^{1}(\Omega)}\big)^{Q/(Q-\alpha)}.
\end{equation}
The weak type estimate implies that the fractional maximal operator maps $L^{1}$ locally to $L^{s}$ whenever $1<s<Q/(Q-\alpha)$.

\begin{corollary}\label{fractional in L1}
Assume that measure lower bound condition \eqref{measure lower bound} holds. 
Let $0<\alpha<Q$ and $1\le s<Q/(Q-\alpha)$. 
If $\Omega\subset X$, $\mu(\Omega)<\infty$ and $u\in L^{1}(\Omega)$,
then $\M_{\alpha,\Omega}u\in L^{s}(\Omega)$ and
\begin{equation}\label{Ls norm}
\|\M_{\alpha,\Omega}u\|_{ L^{s}(\Omega)}
\le C\|u\|_{L^{1}(\Omega)},
\end{equation}
where the constant $C$ depends on the doubling constant, the constant in the measure lower bound, $s$, $\alpha$ and $\mu(\Omega)$. 
\end{corollary}

\begin{proof}
Let $a>0$. Now
\begin{align*}
  \int_{\Omega}(\M_{\alpha,\Omega}u)^{s}\,d\mu
& =s\int_{0}^{\infty}t^{s-1}\mu(\{x\in\Omega:\M_{\alpha,\Omega}u(x)>t\})\,dt\\
&=s\bigg(\int_{0}^{a}+\int_{a}^{\infty}\bigg),
\end{align*}
where 
\[
\int_{0}^{a}t^{s-1}\mu(\{x\in\Omega:\M_{\alpha,\Omega}u(x)>t\})\,dt
\le a^{s}\mu(\Omega).
\]
For the second term, \eqref{weak type in L1} together with the assumption 
$1\le s<Q/(Q-\alpha)$ implies that
\begin{align*}
\int_{a}^{\infty}t^{s-1}\mu(\{x\in\Omega:\M_{\alpha,\Omega}u(x)>t\})\,dt
&\le C\|u\|_{L^{1}(\Omega)}^{Q/(Q-\alpha)}
\int_{a}^{\infty}t^{s-1-Q/(Q-\alpha)}\,dt\\
&=C\|u\|_{L^{1}(\Omega)}^{Q/(Q-\alpha)} a^{s-Q/(Q-\alpha)}.
\end{align*}
Now norm estimate \eqref{Ls norm} follows by choosing $a=\|u\|_{L^{1}(\Omega)}$.
\end{proof}

\subsection{The discrete fractional maximal function}
We begin the construction of the local discrete fractional maximal function in the metric setting
with a Whitney covering as in \cite[Lemma 4.1]{AK}, see also the classical references \cite{CW} and \cite{MS1}.
Let $\Omega\subset X$ be an open set such that $X\setminus\Omega\ne\emptyset$, 
let $0\le\alpha\le Q$ and let $0<t<1$ be a scaling parameter. 
There exist balls $B_{i}=B(x_i,r_{i})$, $i=1,2,\dots$, with $r_{i}=\tfrac1{18}t\dist(x_{i},X\setminus\Omega)$, for which
\[
\Omega=\bigcup_{i=1}^\infty B_{i}
\quad\text{and}\quad
\sum_{i=1}^\infty \ch{6B_{i}}(x)\le N<\infty
\]
for all $x\in\Omega$. The constant $N$ depends only on the doubling constant. Moreover, 
 for all $x\in 6B_{i}$,
\begin{equation}\label{xn etaisyys}
12r_{i}\le t\dist(x,X\setminus\Omega)\le 24r_{i}.
\end{equation}
Using the definition of $r_{i}$, it is easy to show that if $x\in B_{i}$ and $B_{i}\cap 6B_{j}\ne\emptyset$, then 
\begin{equation}\label{ri rj}
r_{i}\le \frac{24}{17}r_{j}\le \frac{3}{2}r_{j}\quad\text{and}\quad
 r_{j}\le \frac{19}{12}r_{i}\le \frac{5}{3}r_{i}.
\end{equation}
Related to the Whitney covering $\{B_{i}\}_i$, 
there is a sequence of Lipschitz functions $\{\ph_{i}\}_i$, called partition of unity, for which
\[
\sum_{i=1}^\infty\ph_i(x)=1
\]
for all $x\in\Omega$. Moreover, for each $i$, the functions $\ph_i$ satisfy the following properties: $0\le\ph_i\le1$, 
$\ph_i=0$ in $X\setminus 6B_{i}$, $\ph_i\ge\nu$ in $3B_{i}$, 
$\ph_i$ is Lipschitz with constant $L/r_i$ where $\nu>0$ and $L>0$ depend only on the doubling constant.

Now the discrete fractional convolution of a locally integrable function $u$ at the scale $t$ is 
$u_{t}^{\alpha}$,
\[
u_{t}^{\alpha}(x)=\sum_{i=1}^{\infty}\ph_{i}(x)r_{i}^{\alpha}{}u_{3B_{i}},\quad x\in X.
\]
Let $t_{j}$, $j=1,2,\dots$ be an enumeration of the positive rationals of the interval $(0,1)$. 
For every scale $t_{j}$, choose a covering of $\Omega$ and a partition of unity as above.
The local discrete fractional maximal function of $u$ in $\Omega$ is $\M^{*}_{\alpha,\Omega}u$,
\[
\M^{*}_{\alpha,\Omega}u(x)=\sup_{j}|u|_{t_j}^{\alpha}(x),\quad x\in X.
\] 
For $\alpha=0$, we obtain the local discrete maximal function studied in \cite{AK}.
The construction depends on the choice of the coverings, but the estimates
below are independent of them.

The local discrete fractional maximal function is comparable to the standard local fractional maximal function. 
The proof of the following lemma is similar as for local discrete maximal function and local Hardy-Littlewood maximal function in \cite[Lemma 4.2]{AK}.

\begin{lemma}\label{M M* comparable}
There is a constant $C\ge 1$, depending only on the doubling constant of $\mu$, such that 
 \[
C^{-1}\M_{\alpha,\Omega}^{24}u(x)
 \le \M^{*}_{\alpha,\Omega}u(x)
 \le C\M_{\alpha,\Omega}u(x)
 \]
for every $x\in X$ and for each locally integrable function $u$.
\end{lemma}
Above, $24$ is the constant from \eqref{xn etaisyys} and 
\[
\M_{\alpha,\Omega}^{\beta} u(x)=\sup\,r^{\alpha}\vint{B(x,r)}|u|\,d\mu,
\]
where the supremum is taken over all radii $r$ for which $0<\beta r<\dist(x,X\setminus\Omega)$, is the restricted local fractional maximal function.

Since the discrete and the standard fractional maximal functions are comparable, the integrability estimates hold for the local discrete fractional maximal function as well, see Theorem \ref{fracM bounded} and \eqref{bdd in Lp}.

\subsection{Sobolev boundary values}
In the metric setting, smoothing properties of the discrete fractional maximal operator in the global case have been studied in \cite{HKNT} and of the standard fractional maximal operator $\M_{\alpha}$ in \cite{HLNT}.
In the local case, by \cite[Theorem 5.6]{AK}, the local discrete maximal operator preserves the boundary values in the Newtonian sense, that is, $|u|-\M^{*}_{\Omega}u\in N^{1,p}_{0}(\Omega)$ whenever $u\in N^{1,p}(\Omega)$.
Intuitively, the definition of the fractional maximal function says that it has to be small near the boundary. In Theorem \ref{sobo nolla}, we will show that if $\Omega$ has finite measure, then the local discrete fractional maximal operator maps $L^{p}(\Omega)$-functions to Sobolev functions with zero boundary values.

The next theorem, a local version of \cite[Theorem 6.1]{HKNT}, shows that the local discrete fractional maximal function of an 
$L^{p}$-function has a weak upper gradient and both $\M^{*}_{\alpha,\Omega}u$ 
and the weak upper gradient belong to a higher Lebesgue space than $u$. 

We use the following simple fact in the proof:
Assume that $u_i$, $i=1,2,\dots$, are functions and 
$g_i$, $i=1,2,\dots$, are $p$-weak upper gradients of $u_i$, respectively. 
Let $u=\sup_i u_i$ and $g=\sup_i g_i$.
If $u$ is finite almost everywhere, then $g$ is a $p$-weak upper gradient 
of $u$. For the proof, we refer to \cite{BB}.

 \begin{theorem}\label{Lp Sobo}
Assume that measure lower bound condition \eqref{measure lower bound} holds. 
Let $\Omega\subset X$ be an open set and let $u\in L^{p}(\Omega)$ with $1<p<Q$.
Let $1\le\alpha<Q/p$,  $p^{*}=Qp/(Q-\alpha p)$ and $q=Qp/(Q-(\alpha-1) p)$.
Then C$\M_{\alpha-1,\Omega}u$ is a weak upper gradient of  $\M^{*}_{\alpha,\Omega}u$.
Moreover, 
\[
\|\M^{*}_{\alpha,\Omega}u\|_{L^{p^{*}}(\Omega)}
\le C\|u\|_{L^p(\Omega)}
\quad\text{and}\quad
\|\M_{\alpha-1,\Omega}u\|_{L^q(\Omega)}
\le C \|u\|_{L^p(\Omega)}.
\]
The constants $C>0$ depend only on the doubling constant, the constant in the measure lower bound, 
$p$ and $\alpha$.
\end{theorem}

\begin{proof}
We begin by showing that $C\M_{\alpha-1,\Omega}u$ is a weak upper gradient of $|u|^{\alpha}_{t}$. 
Let $t\in(0,1)\cap\mathbb Q$ be a scale and let $\{B_{i}\}_i$ be a Whitney covering of $\Omega$. 
Since
 \[
 |u|^{\alpha}_{t}(x)=\sum_{j=1}^{\infty} \ph_{j}(x)r_{j}^{\alpha}|u|_{3B_{j}},
 \]
each $\ph_{j}$ is $L/r_{j}$-Lipschitz continuous and has a support in $6B_{j}$, the function
\[
g_{t}(x)=L\sum_{j=1}^{\infty}r_{j}^{\alpha-1}|u|_{3B_{j}}\ch{6B_{j}}(x)
\]
is a weak upper gradient of $ |u|^{\alpha}_{t}$. We want to find an upper bound for $g_{t}$. 
Let $x\in\Omega$ and let $i$ be such that $x\in B_{i}$. 
Then, by \eqref{ri rj}, 
$3B_{j}\subset B(x,4r_{i})\subset 15B_{j}$ whenever $B_{i}\cap 6B_{j}\ne\emptyset$ and hence
\[
|u|_{3B_{j}}\le C\vint{B(x,4r_{i})}|u|\,d\mu.
\] 
The bounded overlap property of the balls $6B_{j}$ together with estimate \eqref{ri rj} implies that
\[ 
g_{t}(x)\le Cr_{i}^{\alpha-1}\vint{B(x,4r_{i})}|u|\,d\mu
\le C\M_{\alpha-1,\Omega}u(x).
\]
Consequently, $C\M_{\alpha-1,\Omega}u$ is a weak upper gradient of $ |u|^{\alpha}_{t}$. 

By \eqref{bdd in Lp}, the function $\M^{*}_{\alpha,\Omega}u$ belongs to $L^{p^{*}}(\Omega)$ and hence it is finite almost everywhere.
As 
 \[
 \M^{*}_{\alpha,\Omega}u(x)=\sup_{j} |u|^{\alpha}_{t_{j}}(x),
 \]
and because $C\M_{\alpha-1,\Omega}u$ is an upper gradient of $|u|_{t_j}^\alpha$ for every $t=1,2,\dots$,
we conclude that it is an upper gradient of $\M^*_{\alpha,\Omega} u$ as well.
The norm bounds follow from Lemma \ref{M M* comparable} and \eqref{bdd in Lp}.
\end{proof}

\begin{remark}\label{local sobo}
 With the assumptions of Theorem \ref{Lp Sobo}, 
 $\M^{*}_{\alpha,\Omega}u\in N^{1,q}_{\text{loc}}(\Omega)$ and 
 \[
 \|\M^{*}_{\alpha,\Omega}u\|_{N^{1,q}(A)}
 \le C\mu(A)^{1/q-1/p^{*}}\|u\|_{L^p(A)}
 \]
for all open sets $A\subset\Omega$ with $\mu(A)<\infty$.
\end{remark}

\begin{remark}\label{L1 sobo}
Similar arguments as in the proof of Theorem \ref{Lp Sobo} together with  Corollary \ref{fractional in L1} show that if the measure lower bound condition holds, $\Omega\subset X$ is an open set, $u\in L^{1}(\Omega)$, $\mu(\Omega)<\infty$,  
and $1\le s'\le s<Q/(Q-(\alpha-1))$, then $C \M_{\alpha-1,\Omega}u$ is a weak upper gradient of  
$\M^{*}_{\alpha,\Omega}u$ and 
 \[
 \|\M^{*}_{\alpha,\Omega}u\|_{L^{s}(\Omega)}
 \le C\|u\|_{L^1(\Omega)}
 \quad\text{and}\quad
 \|\M_{\alpha-1,\Omega}u\|_{L^{s'}(\Omega)}
 \le C \|u\|_{L^1(\Omega)}.
 \]
In particular, we have that 
 $\M^{*}_{\alpha,\Omega}u\in N^{1,s'}(\Omega)$ and 
 \[
 \|\M^{*}_{\alpha,\Omega}u\|_{N^{1,s'}(\Omega)}
 \le C\mu(\Omega)^{1/s'-1/s}\|u\|_{L^1(\Omega)}.
 \]
 \end{remark}

The next result shows that the local discrete fractional maximal operator actually maps $L^{p}(\Omega)$ to the Sobolev space with zero boundary values.  
\begin{theorem}\label{sobo nolla}
Assume that measure lower bound condition \eqref{measure lower bound} holds and that 
$X$ supports a $(1,p)$-Poincar\'e inequality with $1<p<Q$.  
Let $\Omega\subset X$ be an
open set with $\mu(\Omega)<\infty$ and let $u\in L^{p}(\Omega)$.
Let $1\le\alpha<Q/p$ and  $q=Qp/(Q-(\alpha-1) p)$.
Then $\M^{*}_{\alpha,\Omega}u\in N^{1,q}_{0}(\Omega)$.
\end{theorem}

\begin{proof}
 Let $u\in L^{p}(\Omega)$. 
 By Remark \ref{local sobo}, $\M^{*}_{\alpha,\Omega}u\in N^{1,q}(\Omega)$ and hence, 
 by Theorem \ref{hardy sobo nolla}, it suffices to show that 
 \begin{equation}\label{hardy}
 \int_{\Omega}\bigg(
 \frac{\M^{*}_{\alpha,\Omega}u(x)}{\dist(x,X\setminus\Omega)}\bigg)^{q}\,d\mu(x)<\infty.
\end{equation}
 We begin by considering $|u|^{\alpha}_{t}$.
 Let $t\in(0,1)\cap\mathbb Q$ be a scale and let $\{B_{i}\}_i$ be a Whitney covering of $\Omega$. 
 Let $x\in\Omega$ and let $i$ be such that $x\in B_{i}$. 
Now
 \[
 |u|^{\alpha}_{t}(x)=\sum_{j} \ph_{j}(x)r_{j}^{\alpha}|u|_{3B_{j}},
 \]
where the sum is over such indices $j$ for which $B_{i}\cap 6B_{j}\ne\emptyset$. 
As in the proof of Theorem \ref{Lp Sobo}, we use \eqref{ri rj}, the doubling property, 
the bounded overlap of the balls $B_{j}$ and \eqref{xn etaisyys} to obtain that 
\[
|u|_{3B_{j}}\le C\vint{B(x,4r_{i})}|u|\,d\mu 
\] 
for all such $j$, and that
\[
 |u|^{\alpha}_{t}(x)
 \le Cr_{i}^{\alpha}\vint{B(x,4r_{i})}|u|\,d\mu
 \le C\dist(x,X\setminus\Omega)\M_{\alpha-1,\Omega}u(x).
 \]
 By taking the supremum on the left side we have
 \[
 \M^{*}_{\alpha,\Omega}u(x)
 \le C\dist(x,X\setminus\Omega)\M_{\alpha-1,\Omega}u(x).
 \]
This together with \eqref{bdd in Lp} implies that
\[
\int_{\Omega}\bigg(
\frac{\M^{*}_{\alpha,\Omega}u(x)}{\dist(x,X\setminus\Omega)}\bigg)^{q}\,d\mu(x)
\le C\int_{\Omega}\big(\M_{\alpha-1,\Omega}u\big)^{q}\,d\mu
\le C\|u\|_{L^{p}(\Omega)}^{q}.
\]
Hence \eqref{hardy} holds and the claim follows.
\end{proof}

\begin{remark}
The same proof using Remark \ref{L1 sobo} and norm estimate \eqref{Ls norm} gives a corresponding result for $p=1$. 
Namely, if $u\in L^{1}(\Omega)$, $\mu(\Omega)<\infty$,   $1<\alpha<Q$, and 
$1<s'<Q/(Q-(\alpha-1))$, then $\M^{*}_{\alpha,\Omega}u\in N^{1,s'}_{0}(\Omega)$.
\end{remark}

\vspace{0.5cm}
\noindent
\small{\textsc{Toni Heikkinen},}
\small{\textsc{Department of Mathematics},}
\small{\textsc{P.O. Box 11100},}
\small{\textsc{FI-00076 Aalto University},}
\small{\textsc{Finland}}\\
\footnotesize{\texttt{toni.heikkinen@aalto.fi}}

\vspace{0.3cm}
\noindent
\small{\textsc{Juha Kinnunen},}
\small{\textsc{Department of Mathematics},}
\small{\textsc{P.O. Box 11100},}
\small{\textsc{FI-00076 Aalto University},}
\small{\textsc{Finland}}\\
\footnotesize{\texttt{juha.k.kinnunen@aalto.fi}}

\vspace{0.3cm}
\noindent
\small{\textsc{Janne Korvenp\"a\"a},}
\small{\textsc{Department of Mathematics},}
\small{\textsc{P.O. Box 11100},}
\small{\textsc{FI-00076 Aalto University},}
\small{\textsc{Finland}}\\
\footnotesize{\texttt{janne.korvenpaa@aalto.fi}}

\vspace{0.3cm}
\noindent
\small{\textsc{Heli Tuominen},}
\small{\textsc{Department of Mathematics and Statistics},}
\small{\textsc{P.O. Box 35},}
\small{\textsc{FI-40014 University of Jyv\"askyl\"a},}
\small{\textsc{Finland}}\\
\footnotesize{\texttt{heli.m.tuominen@jyu.fi}}

\end{document}